\documentclass[11pt,oneside,a4paper]{article}
\usepackage[margin=1in]{geometry}
\usepackage{csquotes}
\usepackage{latexsym,upref,amsmath,amssymb,amsfonts,amsthm}
\usepackage{xcolor}
\usepackage[pdfusetitle]{hyperref}
\usepackage{cleveref}

\usepackage{float}
\usepackage{enumitem}

\usepackage[hyphenbreaks]{breakurl}
\newcommand\DOI[1]{{\tt DOI:\burlalt{#1}{https://doi.org/#1}}}

\def\d{\mathbf{d}}
\def\D{\mathbf{D}}
\def\k{\mathbf{k}}

\theoremstyle{plain}
\newtheorem{theorem}{Theorem}[section]
\newtheorem*{theorem*}{Theorem}
\newtheorem{lemma}[theorem]{Lemma}
\newtheorem{corollary}[theorem]{Corollary}

\newtheorem{proposition}[theorem]{Proposition}

\theoremstyle{definition}
\newtheorem{definition}[theorem]{Definition}

\usepackage{authblk}
\usepackage{xurl}

\title{\bfseries New results on graph matching from degree preserving growth\footnote{PLE
and TRM were supported in part by the National Research, Development and
Innovation Office --- NKFIH grant SNN 135643, K 132696. SRK and ZT were
supported by the NSF grant IIS-1724297. \\ \dag\ Equal contributions}}
\author[1]{P\'eter L. Erd\H{o}s}
\author[2\dag]{Shubha R. Kharel}
\author[1\dag]{Tam\'as R\'obert Mezei}
\author[2]{Zolt\'an~Toroczkai}

\affil[1]{\small Dept.\ of Combinatorics and applications, Alfr\'ed R\'enyi Inst.\
of Math.\ (HUN-REN),\protect\\ Re\'altanoda utca 13--15, H-1053 Budapest,
Hungary.\protect\\ \texttt{<erdos.peter, mezei.tamas.robert>@renyi.hun-ren.hu}}
\affil[2]{\small Department of Physics, 225 Nieuwland Science Hall, Notre Dame,
IN 46556, USA.\protect\\\texttt{<torol>@nd.edu, shubha.raj.kharel@gmail.com}}

\begin{document}

\maketitle
\begin{abstract}
The recently introduced \emph{Degree Preserving Growth} model (Nature Physics, \DOI{10.1038/s41567-021-01417-7}) uses matchings to insert new vertices of prescribed degrees into the current graph of an ever-growing graph sequence. The process depends both on the size of the largest  available matchings, which is our focus here, as well as on the actual choice of the matching.

First we show that the question whether a graphic degree sequence, extended with a new degree $2\delta$ remains graphic is closely related to the available matchings in the realizations of the sequence. Namely we  prove that the extension problem is equivalent to the existence of a realization of the original degree sequence with a matching of size $\delta$.

Second we present lower bounds for the \emph{forcible matching number} of degree sequences. This number is the size of the maximum matchings in any realization of the degree sequence. We  then study bounds on the size of maximal matchings in \emph{some} realizations of the sequence, known as the \emph{potential matching number}. We also estimate the  minimum size of both the maximal and the maximum matchings, as determined by the degree sequence, independently of graphical realizations. Along this line we answer a question raised by Biedl, Demaine \emph{et al.} (\DOI{10.1016/j.disc.2004.05.003}).

\medskip


\noindent\emph{Classification}  05C70, 05C82
\end{abstract}

\section{Introduction}\label{sec:intro}

The recently introduced  degree-preserving growth (DPG) network evolution dynamics (see~\cite{kharel_degree 2021} and~\cite{erdos_degree 2021}) chooses $k$ pairwise independent edges (a \emph{$k$-matching}, where $k$ may be chosen dynamically), deletes those edges, and connects all the end points of the original edges with a new vertex $v$. This operation is also called \emph{pinching} of edges by the incoming vertex $v$.  The degree of the new vertex is $d(v)=2k$, and the process  keeps the degrees of all the original vertices unchanged, which is the reason for the name \emph{degree-preserving network growth} (DPG). This property is in sharp contrast with previous network growth dynamic models (such as the preferential attachment model, or the Chung-Lu model for generating scale-free networks) where some of the existing nodes must increase their degrees, whenever a new node attaches to them. (Let's remark that the incoming node's degree can also be odd (see \cite{erdos_degree 2021}), but we will not pursue this issue here.)

Clearly, the process depends on the degree $k$ of the incoming vertex, and also on the actual choice of the matchings to be used. For example, the chosen matching may have always the same size, providing regular graph sequence, or the chosen matching can be always maximum ones, providing center-periphery like networks (see  \cite{kharel_degree 2021}). The choice of the matching of given size can also vary (as there can be several matching of a given size, in general). We will not pursue these  aspects of the problem here.

To finish the short description of the DPG  process, it is enough to mention that different strategies for these choices provide different dynamics. The DPG mechanism defines a whole family of network growth models, proving useful in network modeling applications \cite{kharel_degree 2021}. For example it was proved that it can simulate outputs of several, previously known network growth dynamics, including generating scale-free networks, and furthermore, it  can generate most real-world networks precisely, i.e., the exact graph, edge-by-edge  \cite{kharel_degree 2021}. Interestingly, the latter, numerically observed property is in contrast with the fact (proven in \cite{erdos_degree 2021}), that DPG construction is NP-hard, in general. Finally, it is worth mentioning that the DPG process has an interesting application to number theory \cite{prime}.

\bigskip\noindent
In  this  paper we focus  on the study of the available matchings in the different realizations of  a given degree sequence. The largest size of a matching in finite graphs is called the graph's \emph{matching number}, denoted by $\nu(G)$, and it a graph invariant. It is well-known that there exist polynomial-time algorithms  that find a maximum size  matching in any given \emph{simple} graph $G$  (for example the ``blossom'' algorithm of  Edmonds~\cite{edmonds}). However, it is  difficult to estimate  the matching number analytically on the basis of several commonly used graph parameters. For example, two graphs with the same degree sequence may have very different matching numbers (consider the disjoint union of $2k$ triangles and of the cycle $C_{6k}$ of length $6k$, i.e., $\nu=2k$ vs. $\nu=3k$), even by orders of magnitude (see Section~\ref{sec:conclusion}). Furthermore,  local operations (such as manipulating one vertex and its neighbors) on the graph may change this number significantly. For example, in the friendship graph $F_k$ (which is formed by $k$ copies of $K_3$, overlapping in one vertex), the matching number is $\nu(F_k)=k$. However, after one DPG step (using the edges of the maximum matching) we get the complete bipartite graph $K_{2,2k}$ with  $\nu(K_{2,2k}) = 2$.

\medskip\noindent
The goal of this paper is twofold. On the one hand, it is to study the maximal possible matching number among all the possible graphical realizations of a degree sequence (potential matching number), see \Cref{sec:matching}. On the other hand, in \Cref{sec:force} we study the size of maximal matchings of arbitrary graphs with given degree sequences (forcible matching number). These quantities may help to design new graph growth dynamics to achieve networks with predefined structural properties. We will also answer a question,  first raised by Biedl, Demaine \emph{et al.} in~\cite{biedl_tight_2004} (see \Cref{tm:maxbound}). \Cref{sec:maximum} further improves on some of the results in \Cref{sec:force}.

\medskip\noindent
The topic of \emph{potentially} or \emph{forcibly} P-graphic properties is widely studied, but as far as we are aware, not in the context of matchings. For an early survey paper about the general notions see~Rao, \cite{rao_survey_1981}. These notions had most likely been inspired by the work of C.\ Nash-Williams, who introduced the terms of potentially Hamiltonian and forcibly Hamiltonian, already in 1969 (see~\cite{NW69}). However, this article has not been surveyed in~\cite{rao_survey_1981}.

\bigskip\noindent
Next we  fix the exact notion of a \emph{matching} as it will be applied in this paper. Let $G$ be a simple graph.  Consider a non-negative integer-valued function $f(v)$, defined on $V(G)$. A subgraph $F \subseteq G$ is  an $f$-factor if $\deg_{F}(v)=f(v)$ for all $v\in V(G)$. For vertices with $f(v)=0$, the vertex does not belong to $F$. A \emph{1-factor} is a spanning subgraph with all degree-one vertices and thus it forms a \emph{perfect matching} in $G$. When the $f$-values
vary between 0 and 1, then the $f$-factor is a (partial) matching or a $(0,1)$-factor. For simplicity we will use the slightly ambiguous notion of $\delta$-matching for $(0,1)$-factors of $2\delta$ ones in $f$.

\medskip

For a given graph $G$, the existence of a $1$-factor is fully determined by
Tutte's 1-Factor theorem (see~\cite{tutte}). Deciding whether
there exists an $f$-factor in $G$, in general, is also relatively easy: the
problem is equivalent to finding a $1$-factor in a graph derived from $G$ using
\emph{Tutte-gadgets} (see~\cite{tutte}).

\medskip

However, the potentially and forcibly $f$-factor problems in general are not well-understood. The forcibly 1-factor graphic problem has been solved for a long time (see Bondy and Chv\'atal, \cite{bondy_method_1976}). Furthermore, J.\ Petersen proved in 1891 that every even-degree regular graph contains a 2-factor, and thus such degree sequences are forcibly 2-factor graphic (see~\cite{petersen}). In 2012 Bauer, Broersma, van den Heuvel, Kahl and Schmeichel proved a necessary and sufficient condition for degree sequences to be forcibly 2-factor graphic (see~\cite{bauer_degree_2012}). However, the conditions are very complex.

As we already mentioned, the $(0,1)$-factor problem is not well studied. In this paper we provide novel results on both the potentially and forcibly $(0,1)$-factor problems.

\section{Potentially \texorpdfstring{$(0,1)$}{(0,1)}-factor graphical degree sequences}\label{sec:matching}

A sequence $\d$ of $n$ positive integers ($n$-sequence) is \emph{graphic}, if there is a graph with a vertex labelling such that the degree sequence of the graph is $\d$. Such a graph is called a \emph{realization} of the graphic sequence $\d$. Let us recall an algorithmic characterization of graphic sequences, introduced independently by Havel and Hakimi:
\begin{theorem}[Havel, 1955 \cite{H55} and Hakimi, 1962 \cite{H62}]\label{tm:HH}
    There exists a simple graph with degree sequence $d_1 > 0$, $d_2 \ge \ldots \ge d_n \geq 0$ if and only if  the degree sequence $\d'= d_2-1,\ldots,d_{d_1+1}-1,d_{d_1+2},\ldots,d_n$ is graphic.
\end{theorem}
\noindent
\Cref{tm:HH} also provides the simple quadratic HH-algorithm to generate an actual graph with the given degree sequence (if there is any). At the same time it is interesting to remark that the HH-algorithm cannot construct all realizations of a graphic sequence (see, for  example \cite[Figure~1]{kim_degree-based_2009}). To be able to do so, further considerations are necessary, such as the star-constrained graphicality theorem based algorithm~\cite[Theorem~6]{kim_degree-based_2009}.

\subsection{Extended degree sequences and matchings}\label{sec:supp}
Next, we study the conditions under which we can extend a graphic degree sequence with an even degree. Let us consider a non-increasing graphic sequence $\d$ and an integer $\Delta > 0$. We want to find necessary and sufficient conditions to ensure that the \emph{extended} degree sequence $\d \circ\Delta = d_1,\ldots, d_n,\Delta $ is graphic. (By the handshake lemma it is clear that $\Delta$ should be even.) It is also clear that \Cref{tm:HH} provides the following statement.
\begin{lemma}\label{lm:1}
    Given a non-increasing graphic sequence $\d$ and a positive integer $\delta
    \leq n/2$, the extended sequence $\mathbf{D}(2\delta) =  \mathbf{d} \circ (2\delta)$ is graphic if and only if the sequence $ \d' := d_1-1,\ldots,d_{2\delta}-1, d_{2\delta+1},\ldots,d_n$ is also graphic.
\end{lemma}

\noindent Next we show that the graphicality of the extended sequence is closely related to the existence of large enough matchings in some realizations of the graphic sequence. To that end we introduce a number of notions and notations.

Denote by $m(G)=|E(G)|$ the number of edges in $G$ and by $\Delta(G)$ the
maximum degree in $G$. When $G$ is known from the context, we use only $m$ or
$\Delta$. The degree sequence of a graph $G$ is denoted by $\d(G)$, while the
degree of a vertex $v$ in a degree sequence $\d=\d(G)$ is denoted by $d(v)$. The
$i$\textsuperscript{th} element of a degree sequence $\d$ is denoted by
$d_i$.

A \emph{matching} is a set of pairwise independent edges, and a $k$-matching is a matching comprising $k$ edges. A matching is \emph{maximum} if it has the greatest cardinality among all matchings in $G$. We will use the notation $\nu(G)$ for this maximum cardinality matching (recall, this is the \emph{matching number} of $G$). In turn, a matching is \emph{maximal} if it is not a proper subset of another matching. Clearly, all maximum matchings are maximal,  but not vice-versa. Let $M$ be a matching in some graph $G$. As usual, we say that a vertex $v \in V(G)$ is \emph{matched} if there exists an edge in $M$ incident with $v$.

\bigskip
Consider the integer sequence $\k = (k_{1}, k_{2}, \ldots ,k_{n})$, where for
all $i$ we have $k-1 \le k_i \le k$ where $k$ is a natural number. The following, not very well known
theorem was conjectured by Gr\"unbaum, \cite{grunbaum_combinatorial_1970} (Generalized $k$-factor Conjecture), and was first proven in \cite{kundu_k-factor_1973}:
\begin{theorem}[Kundu, 1973]\label{tm:K}
    Let $\d=(d_1, d_2, \ldots, d_n)$ and $(d_1-k_{1}, d_2-k_{2}, \ldots,d_n-k_{n})$ be
    two graphic sequences. Then there exists a realization of $\d$ which  contains a $\k$-factor.
\end{theorem}
\noindent
In the language introduced earlier: it gives a condition for a degree sequence for being potentially $k$-factor graphic. In the following we will consider only $(0,1)$-factors. In 1974, Lov\'asz, independently from \cite{kundu_k-factor_1973}, gave a simple proof for the case $\k= (1,\ldots,1)$ (\cite{lovasz_valencies_1974}).

\medskip

As far as we know, the following application of Kundu's theorem is new.
\begin{theorem}[\textbf{Weak extension condition}]\label{tm:dpgweak}
Given a graphic sequence $\d$ and a positive integer $\delta \leq n/2$, the sequence $\D(2\delta) = \d \circ (2\delta)$ is graphic if and only if the sequence $\d$ has a realization with a matching of size $\delta$.
\end{theorem}
Let's remark that while technically the proof is quite simple, the statement itself is rather surprising, and in the context of the DPG processes it is also meaningful. It supports that using matchings in the degree preserving graph growth process is a rather natural approach.
\begin{proof}
    If a realization $G$ of the graphic sequence $\d$ has a matching $M$ of size $\delta$, then we construct $G^+$ by
    \emph{pinching} the edges of the matching $M$ onto a new vertex $v$: let $V(G^+)=V(G)\cup \{v\}$, and let
    \begin{equation*}
        E(G^+)=E(G)-M+\{vu\ |\ u\in \cup M\}.
    \end{equation*}
    It is easy to see that the degree sequence of $G^+$ is $\d\circ (2\delta)$.

    \medskip

    For the other direction: let $G^+$ be a realization of $\d \circ (2\delta)$ on the vertex set $V^+$ and let $v \in V^+$ a vertex with degree $d(v)=2\delta$. Take the subgraph $G=G^+-v$ (delete the vertex $v$ and its incident edges). Since both $\d$ and $d(G)$ are graphic degree sequences on the vertex set $(V^+ - v)$, and $\d = d(G) + \k$ where $\k$ has $2\delta$ ones and $n-2\delta$ zeros, therefore Kundu's theorem immediately implies that there is a realization of $\d$ on $(V^+ - v)$ which contains a $\delta$-matching on the neighbors of $v$ in $G^+$.
\end{proof}
\noindent Let us recall a well-known fact:
\begin{lemma}\label{true-HH}
    Let $\d$ be a non-increasing positive integer sequence, and let $\d \circ
    (2\delta)$ be graphic. Let $G$ be a realization of $\d\circ (2\delta)$, where $\deg(u)=2\delta$. If $uv_i\in E(G)$ and $uv_j\not\in E(G)$ for some $j<i$, then there exists another realization $G'$ of $\d \circ (2\delta)$ where $uv_j \in E(G')$ and $uv_i\not\in E(G')$, but the rest of the
    neighborhood of $u$ is not changed.
\end{lemma}
\noindent  This statement motivates introducing a partial ordering on all $(0,1)$ sequence $\k$ with $2\delta$ ones:
\begin{description}
  \item[$(*)$] $\k_1 \preceq \k_2$ if and only if the ones in $\k_1$  can be produced from the ones in $\k_2$ by \emph{left-shifting}.
\end{description}
\begin{lemma}\label{lm:left}
    Let $\d$  be a non-increasing positive integer $n$-sequence, and let $\k$ be
    a $(0,1)$ sequence with $2\delta$ ones. If sequence $\d -
    \k$ is graphic, then so is the sequence $\d - \k'$, for all $\k'$ where
    $\k' \preceq \k$.
\end{lemma}
Note, we do not assume that $\d$ is graphic.
\begin{proof}
Let $G$ be a realization of the graphic sequence $\d - \k$. We construct a
new graph $G'$ by adding a new vertex $u$ to $G$ and connecting it to each $v_i$
where $\k(i)=1$. The degree sequence of $G'$ is $\D(2\delta)= \d\circ (2\delta)$.

Let us (iteratively) apply \Cref{true-HH} to $G'$. Therefore, since $\k'\preceq
\k$,
\begin{description}
  \item[$(**)$] there exists another graphic realization $G^*$ of $\D(2 \delta)$ where the
      neighbors of $u$ are those  $v_i$, for which $\k'(i)$ are ones.
\end{description}
We remark that statement $(**)$ is similar to \cite[Lemma 4]{kim_degree-based_2009}. Now deleting $u$ and its edges from $G^*$, we obtain a graph $G^+$ whose degree sequence is $\d - \k'$.
\end{proof}

\begin{corollary}\label{tm:M}
    Let $\d$ be a non-increasing graphic degree sequence. The degree sequence
    $\d$ has a realization with a matching of size $\delta$ if and only if the
    sequence $d_1-1,\ldots,d_{2\delta}-1,d_{2\delta+1},\ldots,d_n$ is graphic.
\end{corollary}
\begin{proof}{($\Rightarrow$)}
Assume there is a realization $G$ of $\d$ which contains a $\delta$-matching.
Delete this matching, obtaining a graph with degree sequence $\d -\k$ where
$\k$ has $2\delta$ ones. Since $\k_0$ with ones on the first $2\delta$ positions is the minimal element in the poset, \Cref{lm:left} proves the statement.\\
{($\Leftarrow$)} Since, by assumption, $\d-\k_0$ is graphic,  Kundu's Theorem~\ref{tm:K} provides a realization with a $\delta$-matching.
\end{proof}
\noindent Now \Cref{lm:1} and \Cref{tm:M} then lead immediately to a strengthening of
\Cref{tm:dpgweak}:
\begin{theorem}[\textbf{Strong extension condition}]\label{tm:dpg}
Given a graphic sequence $\d$ and a positive even integer $\delta \leq n/2$, the
sequence $\D(2\delta) = \d \circ (2\delta)$ is graphic if and only if the sequence
$\d$ has a realization with a matching of size $\delta$ that covers the
vertices with the largest $2\delta$ degrees.
\end{theorem}
\noindent The special case of this result for \emph{perfect matchings} was proved  by Lov\'asz and Plummer in~\cite[Theorem~10.3.3]{plummer_matching_1986}.
The proof of \Cref{tm:M} from \Cref{tm:dpgweak} is clearly not hard. However, one can argue that this statement can play a similar role as the (also simple) Havel-Hakimi theorem.

\begin{corollary}[]\label{lm:2}
Let $\d$ be a non-increasing graphic sequence on $n$ vertices. If the sequence
$\D(2\delta) =  \d \circ (2\delta)$ is graphic, then the sequence $\D(2\delta')$
is also graphic for all positive integers $\delta'$ less than $\delta$.
\end{corollary}
\begin{proof}
    Since sequence $\D(2\delta)$ is
    graphic, by \Cref{tm:dpg}, the graphic sequence $\d$ has a graphic
    realization with $\delta$ independent edges, and thus it also has $\delta'$
    independent edges. The graphicality of $\D(2\delta')$ follows from
    \Cref{tm:dpg}.
\end{proof}
In view of \Cref{lm:2} one can easily find the largest integer $\delta^{*}$
for which $\D(2\delta^*)$ is graphic, by running a binary search on the interval $\delta \in [2,\lfloor n/2 \rfloor]$.

\subsection{Potentially maximum matchings - analytically}\label{sec:EG}
In this subsection we derive the exact size (Theorem \ref{thm:nu_star_exact}) and a lower bound (Theorem \ref{thm:nu_star_lower_bound}) of the potentially maximum matching
of a given degree sequence from the Erdős-Gallai inequalities.
\begin{definition}
Denote by $\mathcal{G}(\d)$  the set of all labelled simple graph realizations of a graphic degree sequence $\d$ and let
\begin{equation}
\nu^{*}(\d) = \max_{G \in{\cal G}(\mathbf{d}) } \nu(G) \label{mm}
\end{equation}
denote the largest matching number among all of its realizations. (Of course $\nu^{*} =\delta^{*}$, where the latter was determined at end of the previous subsection.)
\end{definition}
\noindent In this subsection we are looking for an analytical formula for
$\nu^{*}$ involving as few inequalities as possible. The basis of our study is
the well-known Erd\H{o}s-Gallai theorem:

\begin{theorem}[Erd\H{o}s-Gallai, 1961,~\cite{EG}]\label{thm:EG}
    A non-increasing degree sequence $\d$ on $n$ vertices is graphic if and only if
    $\sum_{i=1}^n d_i$ is even and
    \begin{equation}\label{eq:EG}
        \sum_{i=1}^k d_i\le k(k-1)+\sum_{i=k+1}^n \min\{d_i,k\}
    \end{equation}
    holds for any $1\le k\le n$. The equivalence is true even if~\eqref{eq:EG}
    is required to hold only for $k=n$ and those $k$ that satisfy $d_k>d_{k+1}$.
\end{theorem}
\noindent This theorem can be also derived from Tutte's factor theorem (see~\cite{tutte2}).
\Cref{tm:dpgweak} implies that $\nu^*(\d)$ can be determined in polynomial time
by finding the largest integer $\delta$ for which $\d\circ (2\delta)$ is graphical,
which can easily be checked via the Erd\H{o}s-Gallai inequalities. However, a number
of these inequalities need not be checked in this scenario, as we will prove in
\Cref{thm:nu_star_exact}.
\begin{definition}
    For any non-increasing degree sequence $\d$, let
    \begin{equation*}
        t_{\d}(\delta)=\big|\left\{ i>\delta\ |\
        d_i=d_{\delta}\right\}\big|-\big|\left\{ i<\delta\ |\
        d_i=d_{\delta}\right\}\big|.\\
    \end{equation*}
\end{definition}

\begin{theorem}\label{thm:nu_star_exact}
    For any graphic (non-zero) non-increasing degree sequence $\d$
    \begin{equation}\label{eq:potentially}
        \begin{split}
            \nu^*(\d)=\max\Bigg\{\frac12\delta\in \mathbb{N}\ \Big|\ &\sum_{i=1}^k d_i\le k^2
                +\sum_{i=k+1}^n \min\{d_i-I_{i\le \delta},k\}\ \forall 1\le
                k<\frac12\delta\text{\ and}\\ & \sum_{i=1}^k
                d_i+|\{i>\delta\ |\ d_i=d_{\delta}\}|\le k^2+\\
        & +\sum_{i=k+1}^n \min\{ d_i-I_{d_i=d_{\delta}},k\}\text{\ for
        }k=\delta+t_\d(\delta) \Bigg\}
        \end{split}
    \end{equation}
    where $I_{X}=1$ if $X$ is true, otherwise $I_{X}=0$.
\end{theorem}
\noindent The proof of this complicated formula is based on  algebraic manipulations of the Erd\H{o}s-Gallai's theorem, and it is shown in the Appendix. Note, the remaining inequalities  in \cref{eq:potentially} can be quite complex from an analytical point of view. As the last result in this section, we provide
a simple lower bound for $\nu^*(\d)$.
\begin{theorem}\label{thm:nu_star_lower_bound}
    The size of the potentially maximum matching in a graphic sequence $\d$ is
    \begin{equation}\label{eq:nu_star_lower_bound}
        \nu^*(\d)\ge \min_{k=1,\ldots,n}\left\lfloor k-1+\frac12 \big|\{i \mid k\le
        d_i\le d_{k}\}\big|\right\rfloor
    \end{equation}
\end{theorem}
\begin{proof}
    For convenience, let
    \begin{equation*}
        m^*=\min_{k=0,\ldots,n-1}\left\lfloor k+\frac12 |\{i \mid k+1\le
        d_i\le d_{k+1}\}|\right\rfloor.
    \end{equation*}
    By \Cref{tm:dpgweak}, to deduce \cref{eq:nu_star_lower_bound}, it is sufficient to prove that
    $d\circ (2m^*)$ is graphic. Note, that
    \begin{equation*}
        m^*\le 0+\frac12 |\{i \mid 1\le
        d_i\le d_{1}\}|\le \frac12 n,
    \end{equation*}
    and $m^*$ is a clearly a positive integer. Since $\d$ is graphical, \Cref{thm:EG} applies. First, we distinguish three cases.

    \begin{itemize}
        \item For $k=n$, we have
    \begin{align*}
        \sum_{i=1}^n d_i &\le n(n-1),\\
        \sum_{i=1}^n d_i+2m^*< \sum_{i=1}^n d_i+2n &\le (n+1)n,
    \end{align*}
    i.e., the $(n+1)$\textsuperscript{th} Erdős-Gallai inequality holds for
    $\d\circ (2m^*)$.

        \item If $2m^*\le d_k$, then
    \begin{align*}
        \sum_{i=1}^k d_i &\le k(k-1)+\sum_{i=k+1}^n \min\{d_i,k\},\\
        \sum_{i=1}^k d_i &\le k(k-1)+\min\{2m^*,k\} +\sum_{i=k+1}^n
        \min\{d_i,k\}.
    \end{align*}
    The last inequality means that the $k$\textsuperscript{th} Erdős-Gallai inequality holds for $\d\circ (2m^*)$
    if $2m^*\le d_k$ and $k\le n$.

    \item If $2m^*> d_{k+1}$ and $k<n$ is a jump locus of $\d$ then
    \begin{align*}
        \sum_{i=1}^k d_i &\le k(k-1)+\sum_{i=k+1}^n \min\{d_i,k\},\\
        \sum_{i=1}^k d_i+2k &\le k(k+1) +\sum_{i=k+1}^n\min\{d_i,k\},\\
        \sum_{i=1}^k d_i+2k+|\{i \mid k+1\le d_i\le d_{k+1}\}| &\le
        k(k+1)+\sum_{i=k+1}^n\min\{d_i,k+1\},\\
        \sum_{i=1}^k d_i+2m^* &\le k(k+1)+\sum_{i=k+1}^n\min\{d_i,k+1\}.
    \end{align*}
    The last inequality is the $(k+1)$\textsuperscript{th} Erdős-Gallai inequality associated to $\d\circ (2m^*)$.
    \end{itemize}

    Note, that  if $2m^*\ge d_k$ and $k+1$ is a jump locus of the non-increasing
    version of $\d\circ (2m^*)$, then $k$ is a jump locus of $\d$. Therefore
    we have shown that the $k$\textsuperscript{th} Erdős-Gallai inequality holds
    for $\d\circ (2m^*)$ when
    \begin{itemize}
        \item $k=n+1$,
        \item $1\le k\le n$ and $2m^*\le d_k$,
        \item $2\le k\le n$ and $2m^*> d_k$ and $k$ is a jump locus of the
            non-increasing version of $\d\circ (2m^*)$.
    \end{itemize}
    It remains to show that \cref{eq:EG} holds for $\d\circ (2m^*)$ with $k=1$
    only when $2m^*>d_1$. By taking $k=0$ in the equation defining $m^*$, we get
    \begin{align*}
        m^*\le 0+\frac12 |\{i \mid 1\le
        d_i\le d_{1}\}|\\
        2m^*\le \sum_{i=1}^n \min\{d_i,1\}.
    \end{align*}
    Via \Cref{thm:EG}, this concludes the proof that $\d\circ (2m^*)$ is graphic.
\end{proof}

\section{Forcibly \texorpdfstring{$(0,1)$}{(0,1)}-factor graphic degree sequences}\label{sec:force}

In this section we are looking for conditions on the degree sequence $\d$ which
make it forcibly $\delta$-matching graphic (i.e., each realization of $\d$
should contain a $\delta$-matching).

\subsection{How big must be the maximal matching in any realization of a general degree sequence?}\label{sec:general}
In this subsection we will study the maximal forcible matching graphic property for general degree sequences. Recall, a matching is maximal, if it is not fully contained by another matching.

For any matching $M$ in a graph $G(V,E)$, let $V_{M} \subseteq V$ denote the set of matched vertices and let $U_{M}$ be the set of unmatched vertices. Clearly, $V_{M} \cup U_{M} = V$. Note that $|V_{M}| = 2 \lvert M \rvert$.

\begin{proposition}\label{tm:mv}
For any maximal matching $M$ in graph $G$  with no isolated vertices we have:
\begin{enumerate}[label={\rm
(\roman*)}]
\item
\begin{equation}
\sum_{v \in V_{M}} d(v) \geq  \sum_{u \in U_{M}} d(u) + 2|M|. \label{mun}
\end{equation}
\item
\begin{equation}
\sum_{v \in V_{M}} d(v) \geq  m(G) + |M|. \label{mun1}
\end{equation}
\end{enumerate}
\end{proposition}
\begin{proof}
    \begin{enumerate}[label={(\rm\roman*)}\ ]
        \item The vertices in $U_{M}$ are clearly independent. Furthermore each
            edge incident with $u\in U_M$ must be incident with an $v\in V_M$.
            Finally $M$ is completely within $V_M$.

        \item Follows immediately from \eqref{mun} after adding to both sides $\sum_{v \in V_{M}} d(v)$ and observing that
$\sum_{v \in V_{M}} d(v) + \sum_{v \in U_{M}} d(v)  = \sum_{v \in V} d(v) =  2m(G)$.
\end{enumerate}
\end{proof}
\noindent We can now use this to obtain more concrete lower bounds on the LHS of~\eqref{mun1}. The simplest of these is:
\begin{corollary}[First observed by Biedl, Demaine {\it et. al.} \cite{biedl_tight_2004}]\label{co:m1}
For any maximal matching $M$ of a graph $G$:
\begin{equation}
|M| \geq \frac{m(G)}{2 \Delta(G) - 1}. \label{mun2}
\end{equation}
\end{corollary}
\begin{proof}
Follows immediately from~\eqref{mun1} after noting that $ 2|M| \Delta \geq
\sum\limits_{v \in V_{M}} d(v)$.
\end{proof}
This statement means that if a matching is smaller than the RHS of~\eqref{mun2},
then the matching can be greedily extended to a bigger one. We remark that this
proof is different from the proof in~Biedl~et~al.\ \cite[Theorem~7]{biedl_tight_2004}. The
authors of that paper raised the following question (Section 5, Problem 2):
``What can be said about the size of maximum matchings in graphs? Can we improve on bound (\ref{mun2})?'' We offer answers in \Cref{tm:maxbound} and \Cref{thm:noP3} below.

\begin{theorem}[Maximality-bound]\label{tm:maxbound}
Let $G$ be a graph without isolated vertices and with the non-increasing degree sequence $\d$. For any maximal matching $M$ in $G$, we have
\begin{equation}\label{eq:max}
|M|\ge k^* = \min\left\{ k\in \mathbb{N}\ \Big|\ \sum_{i=1}^{2k} d_i- m(G) -k\ge 0\right\}.
\end{equation}
The degree sequence is forcibly $k^*$-matching graphic.
\end{theorem}
\begin{proof}
	Let $r(k)=\sum_{i=1}^{2k} d_i-m(G) -k$. Since every $d_i\ge 1$ and the
	degree sequence $\d$ is non-increasing, $r(k)$ is strictly monotone
	increasing. From \cref{mun1} it follows that for any maximal matching $M$
	of size $k$, we have
    \begin{equation}
        r(k)\ge \sum_{v\in \cup M}d(v)-m(G) -|M|\ge 0.
    \end{equation}
	In other words, if $r(k)<0$, then there exists a matching of size at
	least $k+1$, which is equivalent to the statement.
\end{proof}
As we will see soon by \Cref{lemma:at_least_half}, this affirmatively answers the question of the authors
of~\cite{biedl_tight_2004}, since this bound is stronger than the one in (\ref{mun2}).

\medskip

We also want to compare the strength of \Cref{tm:maxbound} with other known lower bounds.  However, there are not many such results on the value of $\nu(\d)$ (without additional special structural requirements on $G$, like, e.g.\ being bipartite). We are aware of only two such results. The first one is based on Vizing's seminal result on the \emph{chromatic index}:
\begin{theorem}[Vizing, 1964, \cite{vizing}]
	For any simple graph $G$, the edge-chromatic number satisfies $\chi'(G)\le \Delta(G)+1$.
\end{theorem}
\noindent As it is easy to see, from this one can derive the following lower bound.
\begin{corollary}[Vizing-bound]\label{cor:vizingbound}
\begin{equation}\label{eq:vizing}
\nu(G) \ge \frac{m(G)}{\Delta(G)+1}.
\end{equation}
\end{corollary}
This approximation can be close to the actual value if the degree distribution is concentrated, but  in case of heterogeneous degree sequences this can be vary far from being sharp. The Vizing-bound is better than inequality (\ref{mun2}), but it only applies to maximum cardinality matchings, and not necessarily to all maximal matchings.

\medskip\noindent To describe the other known lower bound we will use the following notation: let  $t_{G}(q)$ be the number of nodes in $G$ whose degree does not exceed $q$.

\begin{theorem}[{\cite[Theorem~4.4]{erdos_degree 2021}}]\label{th:posa}
    Let $G$ be a simple graph on $n$ nodes. Let
    \begin{equation} \label{eq:rG1}
        r(G):=\min \!\left\{ \ell \in \mathbb{Z}^{+}\ :\ \max_{0\le
        q<\frac{n-\ell}{2}}\left(t_G(q)-q+1\right)\le \ell \right\}.
    \end{equation}
    Then $G$ has a matching of size: $\left\lceil{\frac{n-r(G)}{2}}\right \rceil \leq \nu(G)$.
\end{theorem}
\noindent We will call this result \emph{P\'osa-bound}, since in~\cite{erdos_degree 2021} it was proved using P\'osa's seminal Hamiltonian cycle result~\cite{posa_theorem_1962}. A slightly different form of this result was proved  by Chv\'atal and Bondy already in \cite[Theorem 5.1]{bondy_method_1976} (using, essentially, Chv\'atal's Hamiltonian cycle theorem). Inequality~\eqref{eq:rG1} has proven to be very useful for several DPG dynamics (such as the \textbf{linear DPG} and \textbf{MaxDPG}) models. In a wide range of cases it provides much better estimates than the Vizing-bound,~\cref{eq:vizing}.

\medskip

Next we give four toy degree sequence examples, showing that these three results (Theorem \ref{tm:maxbound}, Lemma \ref{cor:vizingbound} and Theorem \ref{th:posa}) can be useful in different cases.   None of the examples are hard to prove, but it is instructive to consider them.

\begin{enumerate}[label={\rm \underline{\textbf{Example \arabic*.}}},left=0pt]
    \item For $\ell$-regular graphic degree sequence $\d$ on $n$ nodes. (Clearly, there are a big many not isomorphic $\ell$-regular graphs.) The Maximality-bound (\Cref{tm:maxbound}) yields
        \begin{equation}
        		\nu(G) \geq \frac{1}{2} \frac{\ell }{2\ell -1} n\;,
        \end{equation}
        while the Vizing-bound (\Cref{cor:vizingbound}) is almost sharp, since a ``typical'' (uniformly random) regular graph has $\nu(G)=n/2-O(\log n)$ with high probability~\cite{anastos_finding_2021}.

    \item The well-known (non-bipartite) \textbf{half-graph} is defined as follows for every
        even $n$: let the set of vertices be the integers $1,\ldots,{n}$, and two
        distinct vertices $i$ and $j$ are joined by an edge $ij$ if and only if
        $i,j\le n/2$ or $i+n/2\le j$. (Clearly, this graph is unique.) The P\'osa-bound (\Cref{th:posa}) gives the correct $\nu=n/2$, while the estimate given by the Vizing-bound is only $\sim n/4$, and the Maximality-bound is also not any better either ($\sim \frac{2-\sqrt{2}}{4}n$).

    \item Next we consider the also well-known \textbf{windmill graph} $\mathit{Wd}(t,\ell)$ where we have $t$ copies of $K_\ell$ cliques, sharing one central vertex. The special case $\ell=3$: $\mathit{Wd}(t,3)$ is called \textbf{friendship graph}. Clearly, the matching number is $\nu(\mathit{Wd}(t,3)) = t = (n-1)/2$ (near perfect matching, with one  unmatched vertex). The Maximality-bound implies that $\nu(\mathit{Wd}(t,3)) \ge \left \lceil \frac{n+3}{6} \right \rceil$.  The Vizing- and P\'osa-estimates are constants.

    \item For a general windmill graph $\mathit{Wd}(t,\ell)$ the Vizing-bound
        yields $\nu(\mathit{Wd}(t,\ell))\ge \frac{\ell-1}{2}$, the P\'osa-bound gives $\nu\ge \ell$, and the Maximality-bound implies $\nu\ge
        \frac{n-t+1}{4}$.
\end{enumerate}

\noindent In \textbf{Example~1}, the Vizing-bound is a factor of 2 better than the
Maximality-bound $k^*$ in \cref{eq:max}. However, this is a worst case scenario, as the next lemma shows.
Moreover, the next lemma also shows that the maximality-bound is better than the
bound of \Cref{co:m1}, i.e., that we answer positively to the question raised by
the authors in~\cite{biedl_tight_2004}.
\begin{lemma}\label{lemma:at_least_half}
    \begin{equation*}
        \frac12\frac{m}{\Delta+1}<\frac{m}{2\Delta-1}\le k^*
    \end{equation*}
\end{lemma}
\begin{proof}
    Recall from the proof of \Cref{tm:maxbound} the following notation: let $r(k)= \sum_{i=1}^{2k} d_i-m(G) -k$ for a non-increasing degree sequence $\d$.
    \begin{align*}
    r\left(\left\lfloor\frac{m}{2\Delta-1}\right\rfloor\right)=\sum_{i=1}^{2\lfloor
    \frac{m}{2\Delta-1}\rfloor}d_i-m-\left\lfloor\frac{m}{2\Delta-1}\right\rfloor&\le
    \\  \le 2\left\lfloor\frac{m}{2\Delta-1}\right\rfloor\cdot(\Delta-1/2)-m&\le 0.
    \end{align*}
    If any of the inequalities is strict, then it follows from the definition
    that $\lfloor\frac{m}{2\Delta-1}\rfloor< k^*$. If every inequality holds with equality,
    then $\frac{m}{2\Delta-1}$ is an integer and $r\left (\frac{m}{2\Delta-1}- 1
    \right)=1-2\Delta<0$, which also implies the statement.
\end{proof}

\subsection{Strengthening the maximality-bound}\label{sec:maximum}
In the remaining part of this note, we will strengthen our maximality-bound for
both maximal and maximum size matchings. Let us begin with \emph{maximal} matchings.

In addition to inequality (\ref{eq:max}), we can study the derived subgraph $G[U,V]$, getting the following inequality system. This leads to a slightly stronger, but a computationally more complicated result than
\Cref{tm:maxbound}.
\begin{lemma}\label{lemma:max-matchingGR}
    Let $M$ be a maximal matching in $G$ and let its degree sequence be $\d$. Then
    \begin{equation}\label{eq:max-matchingGR}
        \forall U\subseteq U_M\quad \sum_{v \in V_{M}} \min\{ d(v)-1,|U|\} \geq
        \sum_{u \in U} d(u).
    \end{equation}
\end{lemma}
\begin{proof}
    Since $M$ is a maximal matching in $G$, $U_M$ must induce an empty graph in $G$.  The number of edges incident on $U$ is counted on the RHS of
    \cref{eq:max-matchingGR}. The set of edges incident on a vertex of $U$ must be  also incident on $V_M$. A vertex $v\in V_M$ is joined to at most $d(v)-1$ vertices of $U$, because the edge of $M$ which is incident on $v$ is not incident on $U_M$. Also, $v$ is joined to at most $|U|$ vertices of $U$. Therefore the RHS of \cref{eq:max-matchingGR} is  bounded by the sum of $\min\{d(v)-1,|U|\}$ for all $v\in V_M$. (This inequality system is very similar to the one, that appears in the well-known Gale-Ryser theorem, which fully describes the graphic bipartite degree sequences (\cite{G57, R57}).
\end{proof}
We have the following result on the forcibly $k$-matching graphic problem.
\begin{theorem}
    The size of every maximal matching $M$ in any realization of a non-increasing
    degree sequence $\d$ is at least
	\begin{equation}\label{eq:max3}
		|M|\ge \ell^*=\min \left\{\ell\ \Big|\ \sum_{i=1}^{2\ell} \min\{ d_i-1,k\} \geq
		\sum_{i=2\ell+1}^{2\ell+k} d_i,\quad\forall k=1,\ldots,n-2\ell\right\}.
	\end{equation}
\end{theorem}
\begin{proof}
    By \Cref{lemma:max-matchingGR},~\cref{eq:max-matchingGR} holds. If $d(v)\le
    d(u)$ for $v\in V_M,u\in U_M$, then we can swap them between the LHS and RHS
    of~\eqref{eq:max-matchingGR}
    and the inequality will still hold. Thus, we have shown that $|M|\ge \ell^*$
    must hold.
\end{proof}
For $r$-regular graphs $|M|\ge
\ell^*=\max\left(\frac{r}{2},\frac{r}{2r-1}\cdot\frac{n}{2}\right)$.

\bigskip\noindent Now we turn our attention to \emph{maximum} matchings. Let one
of them be $M$.
\begin{lemma}
    Let $M$ be a maximum size matching in a graph $G$ with degree sequence $\d$. Then
    \begin{equation}\label{eq:exP3}
\sum_{uv \in M} \max\left(d(u)-1,d(v)-1\right)+\left|\left\{uv\in M\
		|\ d(u)=d(v)=2\right\}\right|\geq  \sum_{w \in U_{M}} d(w)
\end{equation}
\end{lemma}
\begin{proof}
For any $uv\in M$, if there exist two disjoint edges $e,f$ connecting
$V_M$ to $U_M$ that both intersect $uv$, then we may take $M-uv+e+f$ to obtain a
larger matching, contradicting that $M$ is a maximum matching. Therefore if both $u$
and $v$ have neighbors in $U_M$, then there is exactly 1 such neighbor in $U_M$.
In other words, the number of edges induced between $\{u,v\}$ and $U_M$ is
\begin{equation}\label{eq:uv}
	e(G[\{u,v\}, U_M])\le \max\big\{d(u)-1,\quad
    d(v)-1,\quad\min\left\{2,d(u)+d(v)-2\right\}\big\}.
\end{equation}
The $\max$ on the right hand side takes its value from the third argument
exclusively only if $d(u)=d(v)=2$. Summing inequality (\ref{eq:uv}) over every $uv\in M$ we
obtain inequality (\ref{eq:exP3}).
\end{proof}

Using inequality (\ref{eq:exP3}), we can strengthen \Cref{tm:maxbound} as follows:
\begin{theorem}[Lower bound on the matching number]\label{thm:noP3}
	For any graph $G$ with non-increasing degree sequence $\d$ we have
    \begin{equation}\label{eq:nu_exP3}
		\nu(G)\ge \min\left\{ k\ge 0\ \Big|\ \sum_{i=1}^k
		2d_i+\sum_{i=k+1}^{2k} d_i \geq 2m(G)\right\}.
	\end{equation}
\end{theorem}
\begin{proof}
    Let $M$ be a maximum matching in $G$, and let $k=|M|$. Adding $\sum_{w\in V_M} d(w)$ to both sides of inequality (\ref{eq:exP3}) and performing usual algebraic manipulations we get:
    \begin{equation}\label{eq:max2}
    	\sum_{i=1}^k 2d_i+\sum_{i=k+1}^{2k} d_i \geq 2m(G).
    \end{equation}
\end{proof}
Note, inequality (\ref{eq:nu_exP3}) is not always better than the
maximality-bound (\ref{eq:max}). For example, for the friendship graph we have $
k\ge (t+2)/3$. However, for the general windmill-graph $\mathit{Wd}(t,\ell)$ we
have
\begin{equation*}
k\ge \frac{t(\ell-2)+2}{3}=\frac{n-t+1}{3}
\end{equation*}
This is a factor of $\frac43$ larger than the maximality-bound on the same
degree sequence.

\medskip

For the half-graph for an even $n$, we have
\begin{equation*}
   	6kn-5k^2-3k\ge n^2 \quad \Rightarrow \quad \min k\simeq \frac{n}{5}.
\end{equation*}
Finally, for an $\ell$-regular graph $G$, we have $\nu(G)\ge \frac13n$. The
lower bound is sharp for disjoint union of triangles, and it is a factor of
$\frac43$ better than what we obtained from \Cref{tm:maxbound}. In comparison,
\cite[Theorem~B]{henning_tight_2018} proves $\nu(G)\ge\frac{4n-1}{9}$ for any
connected $3$-regular graph $G$ on $n$ vertices. Of course, for disconnected
graphs, the matching number can be lower: for any $5$-divisible $n$, we have
$\nu\left(\frac{n}{5}\times K_5\right)=\frac{2n}{5}$.

\section{An open problem: minimum maximum matching}\label{sec:conclusion}

We finish this paper with an open problem. A graph $G$ can have several maximal matchings. Let us denote the smallest size among the maximal matchings with $\bar{\nu}(G)$ and
thus $\bar{\nu}(G) \leq |M| \leq \nu(G)$ for any maximal matching $M$.  In
general, it was showed in~\cite{yannakakis_edge_1980} that to determine the
value $\bar{\nu}(G)$ is NP-hard. The same is true for several restricted graph
classes. Since $2\bar{\nu}(G)\ge \nu(G)$ therefore a 2-approximation is trivial.
However, Chleb\'\i k and Chleb\'\i kov\'a showed (\cite{chlebik_approximation_2006}), that a $7/6$ approximation algorithm  for the value $\bar{\nu}(G)$ is also NP-hard. Let $\bar\nu(\d)$ denote the minimum possible $\bar\nu(G)$ for all possible realizations of the graphic degree sequence $\d$. It seems to be an interesting question is whether $\bar\nu(\d) = \ell^*(\d)$ (see inequality~\ref{eq:max3}).

\medskip

For any positive integer $t$, let us consider the following degree sequence:
$\mathbf{h}=({(2t-1)}^{2t},1^{2t(2t-1)})$ (the number of vertices is $n=4t^2$). Then~\eqref{eq:max3} gives
\begin{equation*}
    \bar\nu(\mathbf{h})\ge \ell^*(\mathbf{h})=2t
\end{equation*}
Let $G_1=K_{2t}+t(2t-1)\times K_2$ and $G_2=2t\times K_{1,2t-1}$, both realizations of $\mathbf{h}$. We have
$\nu(G_1)=2t^2$ and $\nu(G_2)=2t$, so $\ell^*(\mathbf{h})$ is indeed equal to
$\bar\nu(\mathbf{h})$, even though $\mathbf{h}$ is potentially perfectly
matchable. We also conjecture that
\begin{equation}\label{conj:starbar}
\nu^*(\d)\le\frac12{\bar\nu(\d)}^2
\end{equation}
holds for any degree sequence $\d$. The degree sequence $\mathbf{h}$ shows
that~\eqref{conj:starbar} is potentially sharp.

\bibliographystyle{plain}

\begin{thebibliography}{99}
\bibitem{anastos_finding_2021} M. Anastos - A. Frieze: Finding maximum matchings in random regular graphs in linear expected time, {\sl Random Structures and Algorithms} {\bf 58} (3) (2021), 390--429. \DOI{10.1002/rsa.20980}

\bibitem{bauer_degree_2012} D. Bauer - H.J. Broersma - J. van den Heuvel - N. Kahl - E. Schmeichel: Degree Sequences and the Existence of $k$-Factors, {\sl Graphs and Combinatorics} {\bf 28}(2) (2012), 149--166. \DOI{10.1007/s00373-011-1044-z}

\bibitem{biedl_tight_2004} T. Biedl - E.D. Demaine - C.A. Duncan - R. Fleischer - S.G. Kobourov: Tight bounds on maximal and maximum matchings. {\sl Discr. Math.} {\bf 285} (2004), 7--15. \DOI{10.1016/j.disc.2004.05.003}

\bibitem{bondy_method_1976} J. A. Bondy - V. Chv\'atal: A method in graph theory. {\sl Discrete Mathematics} {\bf 15} (2), (1976), 111-135. \DOI{10.1016/0012-365X(76)90078-9}

\bibitem{chlebik_approximation_2006} M. Chlebl\'ik - J. Chlebl\'ikov\'a: Approximation hardness of edge dominating set problems, {\sl J. Comb. Optim.} {\bf 11} (2006), 279--290. \DOI{10.1007/s10878-006-7908-0}

\bibitem{edmonds} Edmonds J.: Paths, trees, and flowers, {\sl Canad. J. Math.} {\bf 17} (1965), 449--467. \DOI{10.4153/CJM-1965-045-4}

\bibitem{EG} P.~Erd\H{o}s - T.~Gallai: Graphs with prescribed degree of vertices (in Hungarian), {\sl  Mat. Lapok} {\bf 11} (1960), 264--274.

\bibitem{erdos_degree 2021} P.L. Erd\H{o}s - S. Kharel - T.R. Mezei - Z. Toroczkai: Degree preserving graph dynamics - a versatile process to construct random networks, {\sl Journal of Complex Networks}, {\bf 11}(6) (2023),\# cnad046, pp. 34. \DOI{10.1093/comnet/cnad046}

\bibitem{prime} P.L. Erd\H{o}s - G. Harcos - S.R. Kharel - P. Maga - T.R. Mezei - Z. Toroczkai: The sequence of prime gaps is graphic, {\sl Mathematische Annalen} {\bf 388}(2) (2024), 2195 - 2215. \DOI{10.1007/s00208-023-02574-1}

\bibitem{frieze_perfect_2004} A. Frieze - B. Pittel: Perfect Matchings in Random Graphs with Prescribed Minimal Degree, in {\sl Mathematics and Computer Science {III}} Ed. M. Drmota et. al. Trends in Mathematics. Basel: Birkhäuser, (2004), 95-–132. \DOI{10.1007/978-3-0348-7915-6\_11.}

\bibitem{G57} D.~Gale: A theorem on flows in networks, {\sl Pacific J. Math.} {\bf 7}  (2)  (1957), 1073--1082.

\bibitem{H62} S.L.~Hakimi: On the realizability of a set of integers as degrees of the vertices of a simple graph. {\sl J. SIAM Appl. Math.} {\bf 10} (1962), 496--506.

\bibitem{H55} V.~Havel: A remark on the existence of finite graphs. (Czech), {\sl \v{C}asopis P\v{e}st. Mat.} {\bf 80} (1955), 477--480.

\bibitem{grunbaum_combinatorial_1970} B. Gr\"unbaum: Problem 2, Combinatorial Structures and their Applications (Prec. Calgary Internat. Conf., Calgary, Al~a., 1969), Gordon and Breach; New York, 1970, p. 492.

\bibitem{henning_tight_2018} M.A. Henning - A. Yeo: Tight lower bounds on the matching number in a graph with given maximum degree, {\sl Journal of Graph Theory} {\bf 89}(2) (2018), 115--149. \DOI{10.1002/jgt.22244}

\bibitem{kharel_degree 2021} S. Kharel  - T.R. Mezei - S. Chung - P.L. Erd\H{o}s  - Z. Toroczkai: Degree-preserving network growth, {\sl Nature Physics} {\bf 18} (1) (2022), 100--106. \DOI{10.1038/s41567-021-01417-7}

\bibitem{kim_degree-based_2009} Hyunju Kim - Z. Toroczkai - P.L. Erd\H{o}s - I. Mikl\'os - L.A. Sz\'ekely: Degree-based graph construction, {\sl J. Phys. A: Math. Theor.} {\bf 42} (2009) 392001 (10pp) \DOI{10.1088/1751-8113/42/39/392001}

\bibitem{kundu_k-factor_1973} S. Kundu: The $k$-factor conjecture is true. {\sl Discrete Mathematics} {\bf 6} (1973), 367--376. \DOI{10.1016/0012-365X(73)90068-X}

\bibitem{lovasz_valencies_1974} L. Lov\'asz: Valencies of graphs with 1-factors. {\sl Period.Math.Hungar.} {\bf 5} (1974), 149--151. \DOI{10.1007/bf02020548}

\bibitem{plummer_matching_1986} L. Lov\'asz - M.D. Plummer: \textsl{Matching Theory}, Annals of Discrete Mathematics {/bf 29}, Notrh Holland, (1986) pp. 543. {\tt ISBN: 0 444 87916 1}

\bibitem{NW69} Crispin Nash-Williams: Valency Sequences which force graphs to have Hamiltonian Circuits, {\sl Interim Report}, University of Waterloo (1969).

\bibitem{petersen} J.~Petersen: Die Theorie der regul{\" a}ren graphs, {\sl Acta Mathematica} {\bf 15} (1891), 193--220.

\bibitem{posa_theorem_1962} L. P\'osa: A theorem concerning Hamilton lines, {\sl Magyar Tud. Akad. Mat. Kutat\'o Int. K\"ozl.} {\bf 7} (1962), 225--226.

\bibitem{rao_survey_1981} S. B. Rao: A survey of the theory of potentially P-graphic and forcibly $P$-graphic degree sequences, {\sl Lecture Notes in Mathematics} {\bf 885} (1981), 417-–440. \DOI{10.1007/bfb0092288}

\bibitem{R57} H.J.~Ryser: Combinatorial properties of matrices of zeros and ones, {\sl Canad. J. Math.} {\bf 9} (1957), 371--377.

\bibitem{tutte} W.T.~Tutte: The factors of graphs, {\sl Canadian Journal of Mathematics} {\bf 4} (1952), 314--328.

\bibitem{tutte2} W.T.~Tutte: Graph factors, {\sl Combinatorica} {\bf 1} (1981), 79-–97. \DOI{10.1007/BF02579180}

\bibitem{vizing} Vizing V.G.:  On an estimate of the chromatic class of a $p$-graph. {\sl Diskret. Analiz.} {\bf 3}(1964),  25--30.

\bibitem{yannakakis_edge_1980} M. Yannakakis - F. Gavril: Edge Dominating Sets in Graphs, {\sl SIAM J. Appl. Math.} {\bf 38}(3) (1980), 364--372. \DOI{10.1137/0138030}
\end{thebibliography}

\section*{Appendix}
In this section, for sake of completeness, we provide the proof of Theorem \ref{thm:nu_star_exact}.

\begin{theorem*}
    For any graphic (non-zero) non-increasing degree sequence $\d$
    \begin{equation}\label{eq:potentially1}
        \begin{split}
            \nu^*(\d)=\max\Bigg\{\frac12\delta\in \mathbb{N}\ \Big|\ &\sum_{i=1}^k d_i\le k^2
                +\sum_{i=k+1}^n \min\{d_i-I_{i\le \delta},k\}\ \forall 1\le
                k<\frac12\delta\text{\ and}\\ & \sum_{i=1}^k
                d_i+|\{i>\delta\ |\ d_i=d_{\delta}\}|\le k^2+\\
        & +\sum_{i=k+1}^n \min\{ d_i-I_{d_i=d_{\delta}},k\}\text{\ for
        }k=\delta+t_\d(\delta) \Bigg\}
        \end{split}
    \end{equation}
    where $I_{X}=1$ if $X$ is true, otherwise $I_{X}=0$.
\end{theorem*}
\begin{proof}
    By \Cref{tm:M}, $\nu^*(\d)$ is equal to one half of the largest even number $\delta$ for which the
    reduced degree
    sequence~$\d'=(d_1-1,\ldots,d_{\delta}-1,d_{\delta+1},\ldots,d_n)$ is
    graphic. Let $\d''$ denote the non-increasing version of~$\d'$.

    Observe that if $d_{\delta}>d_{\delta+1}$, then $\d''=\d'$ and any jump
    locus of $\d''$ is a jump locus of $\d$.

    If $d_{\delta}=d_{\delta+1}$, then $d'_{\delta}=d'_{\delta+1}-1$ and $\d''$ can be obtained from $\d'$ by
    transposing two contiguous blocks of degrees equal to $d_{\delta}$ and
    $d_{\delta}-1$, respectively.
    Let $k_1$ be the largest integer such that $d_{k_1+1}=d_\delta$. Let $k_2$ be the largest integer such that
    $d_{k_2}=d_\delta$. Then $k_1+(k_2-\delta)$ is a possible jump locus of
    $\d''$. However, any other jump locus of $\d''$ is also a jump locus of $\d$. Note, that
    \begin{equation*}
        \delta+t_{\d}(\delta)=\delta+(k_2-\delta)-(\delta-k_1)=k_1+k_2-\delta.
    \end{equation*}
    Restating our previous observation, if $d''_k>d''_{k+1}$ and $d_k=d_{k+1}$,
    then we must have $k=\delta+t_\d(\delta)$.

    From \Cref{thm:EG} it follows that $\d''$ (and $\d'$) is graphic if and only
    if $\d''$ satisfies the Erdős-Gallai inequalities for $k=n$, $k=\delta+t_\d(\delta)$,
    and whenever $d_k>d_{k+1}$.

    For $k=\delta+t_\d(\delta)$, the Erdős-Gallai inequality for $\d''$ requires:
    \begin{equation}
        \sum_{i=1}^k d_i-k+(k_2-\delta)\le
        k(k-1)+\sum_{i=k+1}^n \min\{ d_i-I_{i\le k_2},k\}
            \quad\left(\text{for }k=\delta+t_\d(\delta)\right).\label{eq:tdk1}
    \end{equation}
    For $k\neq\delta+t_\d(\delta)$, if $k$ is a jump locus of $\d''$ (and $\d$) or
    $k=n$, we have to ensure that
    \begin{align}
        \sum_{i=1}^k d_i-k\le k(k-1)+\sum_{i=k+1}^\delta \min\{
            d_i-1,k\}+\sum_{i=\delta+1}^n\min\{d_i,k\} &\qquad\text{if }1\le
            k<\delta+t_\d(\delta),\label{eq:smallk1}\\
        \sum_{i=1}^k d_i-\delta\le k(k-1)+\sum_{i=k+1}^n \min\{d_i,k\} &
            \qquad\text{if }\delta+t_\d(\delta)< k\le n\label{eq:largek1}.
     \end{align}
     \Cref{eq:largek1} automatically follows from the Erdős-Gallai inequality
     for $\d$ and $k$. Thus the reduced degree sequence
     $\d'$ is graphic if and only if \cref{eq:smallk1,eq:tdk1} are satisfied.
     Inequality~\eqref{eq:smallk1} follows from the graphicality of $\d$ if
     \begin{equation}\label{eq:smallk_sufficient1}
         k\ge \sum_{i=k+1}^\delta \min\{
            d_i-1,k\}-\sum_{i=k+1}^\delta \min\{
            d_i,k\}\ge \left|\left\{ i\in\mathbb{N}\ |\ k+1\le i\le \delta\text{\ and }
         d_i\le k\right\} \right|.
     \end{equation}
     In particular, if $k\ge \frac12\delta$, then \cref{eq:smallk_sufficient1}
     holds. In turn, \cref{eq:smallk1} is
     automatically satisfied if $k\ge \frac12\delta$, which leads to \cref{eq:potentially1}.
\end{proof}

\end{document}